\documentclass{amsart}

\newtheorem{theorem}{Theorem}[section]
\newtheorem{proposition}[theorem]{Proposition}
\newtheorem{corollary}[theorem]{Corollary}
\newtheorem{lemma}[theorem]{Lemma}

\theoremstyle{definition}

\newtheorem{problem}[theorem]{Problem}

\theoremstyle{remark}

\numberwithin{equation}{section}

\def\ldiv{\backslash}
\def\rdiv{/}
\def\mlt#1{\mathrm{Mlt}(#1)}
\def\inn#1{\mathrm{Inn}(#1)}
\def\aut#1{\mathrm{Aut}(#1)}
\def\A{\mathbf A}
\def\B{\mathbf B}
\def\Q{\mathbf Q}
\def\G{\mathbf \Gamma}
\def\V{\mathbf V}
\def\Z{\mathbb Z}
\def\setof#1#2{\{#1\,|\,#2\}}

\def\genof#1#2{\langle#1\,|\,#2\rangle}

\begin{document}

\title[Latin quandles and Bruck loops]{Enumeration of involutory latin quandles, Bruck loops and commutative automorphic loops of odd prime power order}

\author{Izabella Stuhl}
\address{Department of Mathematics, Penn State University, 54 McAllister Street, University Park, State College, PA 16801, U.S.A.}
\email{ius68@psu.edu}

\author{Petr Vojt\v{e}chovsk\'y}
\address{Department of Mathematics, University of Denver, 2390 S York Street, Denver, Colorado 80208, U.S.A.}
\email{petr@math.du.edu}
\thanks{Petr Vojt\v{e}chovsk\'y partially supported by the University of Denver PROF grant.}


\date{October 17, 2017}

\begin{abstract}
There is a one-to-one correspondence between involutory latin quandles and uniquely $2$-divisible Bruck loops. Bruck loops of odd prime power order are centrally nilpotent. Using linear-algebraic approach to central extensions, we enumerate Bruck loops (and hence involutory latin quandles) of order $3^k$ for $k\le 5$, except for those loops that are central extensions of the cyclic group of order $3$ by the elementary abelian group of order $3^4$.

Among the constructed loops there is a Bruck loop of order $3^5$ whose associated $\Gamma$-loop is not a commutative automorphic loop. We independently enumerate commutative automorphic loops of order $3^k$ for $k\le 5$, with the same omission as in the case of Bruck loops.
\end{abstract}


\maketitle

\section{Introduction}

Quandles are self-distributive algebras designed for colorings of arcs of oriented knot diagrams \cite{Joyce,Matveev}. The standard axioms of quandles give sufficient (and in some sense necessary) conditions for the arc colorings to be invariant under Reidemeister moves. Quandles also form a class of set-theoretical solutions of the quantum Yang-Baxter equation \cite{Drinfeld, Eisermann}. Of particular interest in both areas are connected quandles, which are quandles whose left translations generate a permutation group that acts transitively on the underlying set \cite{HSV}. Latin quandles form a proper subset of connected quandles and, being quasigroups, they can be investigated not only by methods of quandle theory but also by methods of quasigroup theory \cite{Galkin}. For an introduction to quandles, see \cite{EN}.

Bruck loops (also known as $K$-loops) form a well-studied variety of loops with properties close to abelian groups \cite{Kiechle}. In his seminal work \cite{Glauberman1, Glauberman2}, Glauberman derived many structural results for Bruck loops in which all elements have a finite odd order, and then went on to transfer some of these results to Moufang loops. For instance, he proved that every Bruck loop of odd order and every Moufang loop of odd order are solvable \cite{Glauberman2}. Bruck loops are also important in the study of neardomains \cite{Kiechle} and show up as a natural algebraic structure describing relativistic addition of vectors \cite{Ungar}. For an introduction to loop theory, see \cite{Bruck,Pflugfelder}.

\medskip

There is a one-to-one correspondence between involutory latin quandles and uniquely $2$-divisible Bruck loops (see Theorem \ref{Th:QBCorrespondence}). Therefore, one can construct all involutory latin quandles of a given odd order $n$ by constructing all Bruck loops of order $n$. This observation is of importance for two reasons:

First, while the theory of extensions is not well developed for involutory latin quandles, Glauberman showed \cite{Glauberman1} that any Bruck loop of odd prime power order $p^k$ is centrally nilpotent and hence is a central extension of the $p$-element group $\Z_p$ by some Bruck loop of order $p^{k-1}$. Given a Bruck loop $F$ of order $p^{k-1}$, the vector space of cocycles that yield all central extensions of $\Z_p$ by $F$ as well as the subspace of coboundaries can be calculated by solving a certain system of linear equations over the $p$-element field, cf. Corollary \ref{Cr:LeftBruckCocycle}

Second, any enumeration of involutory latin quandles or Bruck loops of order $p^k$ will invariably require explicit isomorphism checks, given that the isomorphism problem for central extensions is not solved. Every latin quandle is homogeneous, that is, its automorphism group acts transitively on the underlying set, while Bruck loops are certainly not homogeneous. It is therefore possible in principle and in practice to partition a given Bruck loop into nontrivial blocks that are preserved under isomorphisms, thus greatly aiding in isomorphism searches.

Since Bruck loops are precisely the Bol loops satisfying the automorphic inverse property, results on Bol loops are relevant here. Let $p$ be a prime. Bol loops are power associative and hence all Bol loops of order $p$ are groups. Burn showed \cite{Burn} that all Bol loops of order $p^2$ are groups as well. However, there exist nonassociative Bruck loops of order $p^3$. Using central extensions, we construct here all Bruck loops (and hence all involutory latin quandles) of orders $3^k$ for $k\le 4$, and also all Bruck loops of order $3^5$ that are not central extensions of $\Z_3$ by the elementary abelian group of order $3^4$. The results can be found in Theorem \ref{Th:MainEnum} and Table \ref{Tb:3}. Computationally speaking, the task is not difficult for $n=3^k$ with $k\le 4$, but it takes several months of computing time for $n=3^5$, and the excluded case of order $3^5$ is out of reach because the corresponding vector space of cocycles modulo coboundaries has dimension $24$.

\medskip

All left translations of quandles are automorphisms, and so are all left inner mappings of left Bruck loops. Automorphic loops \cite{BP}, which are loops whose inner mappings are automorphisms, are therefore of interest here. Like Bruck loops, commutative automorphic loops of odd prime power order are centrally nilpotent \cite{JKVNilp}. Moreover, there is a one-to-one correspondence between Bruck loops of odd order and the so-called $\Gamma$-loops of odd order \cite{Greer}, a class of loops that contains all commutative automorphic loops of odd order.

It was known that there exist $\Gamma$-loops of odd order that are not commutative automorphic loops \cite{Greer}, but it was not known until now if there exists a $\Gamma$-loop of odd prime power order that is not a commutative automorphic loop. By exhaustively inspecting the Bruck loops obtained in our enumeration, we found several Bruck loops of order $3^5$ whose corresponding $\Gamma$-loops are not commutative automorphic loops. (At the moment we do not understand the abstract reason for the existence of these examples.)

Finite commutative automorphic loops are solvable; see \cite{JKVStruct} for the odd case and \cite{GKN} for the general case. Automorphic loops of odd order are solvable \cite{KKPV}. Let $p$ be a prime. Since automorphic loops are power associative \cite{BP}, automorphic loops of order $p$ are groups. Cs\"org\H{o} showed \cite{Csorgo, KKPV} that all automorphic loops of order $p^2$ are groups as well. There exist commutative automorphic loops of order $8$ that are not centrally nilpotent. Not all automorphic loops of odd order $p^3$ are centrally nilpotent and their classification is open \cite{KKPV}. Commutative automorphic loops of order $p^3$ have been first obtained in \cite{JKVConst} and classified in \cite{DGV}---there are $7$ commutative automorphic loops of order $p^3$, independent of the prime $p$.

Due to the newly discovered examples, the enumeration of commutative automorphic loops of odd prime power order does not coincide with the enumeration of Bruck loops of odd prime power order. The former enumeration can either  be obtained from the latter by constructing all associated $\Gamma$-loops and checking whether their inner mappings are automorphisms, or by using central extensions for commutative automorphic loops, cf. Corollary \ref{Cr:CommutativeAutomorphicCocycle}. We have opted for the second method and enumerated all commutative automorphic loops of orders $3^k$ for $k\le 4$, and also all commutative automorphic loops of order $3^5$ that are not central extensions of $\Z_3$ by the elementary abelian group of order $3^4$. The results can again be found in Theorem \ref{Th:MainEnum} and Table \ref{Tb:3}.

\begin{theorem}\label{Th:MainEnum}
Up to isomorphism, there are $7$ left Bruck loops (equivalently, involutory latin quandles) of order $3^3$, $72$ of order $3^4$, and $118673$ of order $3^5$, excluding central extensions of $\Z_3$ by $\Z_3^4$.

Up to isomorphism, there are $7$ commutative automorphic loops of order $3^3$, $72$ of order $3^4$, and $118405$ of order $3^5$, excluding central extensions of $\Z_3$ by $\Z_3^4$.
\end{theorem}

\section{Notation and background material}\label{Sc:Notation}

In this section we gather required definitions and background material. The results presented in this section will be used throughout the paper, often without warning.

Let $(Q,\cdot)$ be a groupoid. For every $x\in Q$, let $L_x:Q\to Q$, $y\mapsto x\cdot y$ be the \emph{left translation by $x$}, and $R_x:Q\to Q$, $y\mapsto y\cdot x$ the \emph{right translation by $x$}.

A \emph{left quasigroup} is a groupoid in which all left translations are bijections. In a left quasigroup, we let $x\ldiv y = L_x^{-1}(y)$ be the left division operation. Similarly, a \emph{right quasigroup} is a groupoid in which all right translations are bijections, and then we denote by $y\rdiv x = R_x^{-1}(y)$ the right division operation. A \emph{quasigroup} is a groupoid in which all translations are bijections.

We will often use juxtaposition in place of the multiplication operation and we introduce the following priority rules for expressions involving multiplications and divisions: juxtaposition is more binding than divisions, which are in turn more binding than multiplication. For instance, $x\cdot yz\ldiv u$ means $x((yz)\ldiv u)$.

A groupoid $Q$ is \emph{left involutory} if $L_x^2=1$ for every $x\in Q$. From the condition $L_x^2=1$ we deduce that $L_x$ is a bijection and $L_x^{-1}=L_x$. In other words, every left involutory groupoid is a left quasigroup satisfying $xy = x\ldiv y$. If $Q$ is a left involutory quasigroup, then $x=(x/y)y = (x/y)\ldiv y$, so $y=(x/y)x$ and $y/x = x/y$.

A \emph{loop} is a quasigroup with an identity element, usually denoted by $e$. In a loop $Q$, we say that an element $x\in Q$ has a \emph{two-sided inverse} if there is $x^{-1}\in Q$ such that $xx^{-1}=x^{-1}x=e$. A loop with two-sided inverses has the \emph{left inverse property} if $x^{-1}(xy)=y$ holds, and the \emph{automorphic inverse property} if $(xy)^{-1}=x^{-1}y^{-1}$ holds. A quasigroup is \emph{power associative} if every element generates a group, \emph{flexible} if $x(yx)=(xy)x$ holds, and \emph{left power alternative} if it is power associative and $x^n(x^my) = x^{n+m}y$ holds for all integers $n$, $m$.

A \emph{left Bol loop} is a loop satisfying the left Bol identity
\begin{equation}\label{Eq:LeftBol}
    x(y(xz)) = (x(yx))z.
\end{equation}
Left Bol loops are power associative and left power alternative (hence have the left inverse property). A \emph{left Bruck loop} is a left Bol loop with the automorphic inverse property.

In a quasigroup $Q$ with inverses, let $P_x=L_{x^{-1}}^{-1}R_x$. A \emph{$\Gamma$-loop} is a commutative loop with automorphic inverse property satisfying $L_xL_{x^{-1}}=L_{x^{-1}}L_x$ and $P_xP_yP_x = P_{P_x(y)}$. Every $\Gamma$-loop is power associative.

For a loop $Q$ let $\mlt{Q}=\genof{L_x,R_x}{x\in Q}$ be the \emph{multiplication group} of $Q$, and $\inn{Q}=\setof{\varphi\in\mlt{Q}}{\varphi(e)=e}$ the \emph{inner mapping group} of $Q$. It is well-known that $\inn{Q}=\genof{T_x,L_{x,y},R_{x,y}}{x,y\in Q}$, where
\begin{displaymath}
    T_x=L_x^{-1}R_x,\ L_{x,y} = L_{yx}^{-1}L_yL_x, \text{ and } R_{x,y} = R_{xy}^{-1}R_yR_x.
\end{displaymath}
The \emph{center} $Z(Q)$ is the set of all elements $x\in Q$ such that $\varphi(x)=x$ for all $\varphi\in\inn{Q}$. A loop $Q$ is \emph{centrally nilpotent} if the sequence
\begin{displaymath}
Q,\ Q/Z(Q),\ (Q/Z(Q))/Z(Q/Z(Q)),\ \dots
\end{displaymath}
terminates at the trivial loop in finitely many steps.

A loop $Q$ is \emph{automorphic} if $\inn{Q}\le\aut{Q}$. Note that a commutative loop $Q$ is automorphic if and only if $L_{x,y}\in\aut{Q}$ for every $x$, $y\in Q$.

A groupoid $Q$ is \emph{uniquely $2$-divisible} if the squaring map $x\mapsto x^2$ is a bijection of $Q$. A finite left Bruck loop is uniquely $2$-divisible if and only if it is of odd order. If $Q$ is a left Bruck loop in which every element has finite odd order, then $\mlt{Q}$ is uniquely $2$-divisible. (There are examples of infinite uniquely $2$-divisible left Bruck loops $Q$ for which $\mlt{Q}$ is not uniquely $2$-divisible.)

Let $(Q_1,\cdot)$, $(Q_2,\circ)$ be groupoids. A triple $(\alpha,\beta,\gamma)$ of bijections $Q_1\to Q_2$ is an \emph{isotopism} from $(Q_1,\cdot)$ onto $(Q_2,\circ)$ if $\alpha(x)\circ\beta(y)=\gamma(x\cdot y)$ for every $x$, $y\in Q_1$, in which case we say that the two groupoids are \emph{isotopic}. A groupoid isotopic to a quasigroup is itself a quasigroup. If $Q_1=Q_2$ and the above isotopism of quasigroups has the form $(R_a,L_b,1)$ for some $a$, $b\in Q_1$, then $(Q_2,\circ)$ is a loop with identity element $b\cdot a$.

A \emph{quandle}\footnote{It would make a lot of sense to call quandles \emph{left quandles} but traditionally the chirality of the quandle is suppressed in its name.} is a left quasigroup that is left distributive (that is, the identity $x(yz)=(xy)(xz)$ holds) and idempotent (that is, the identity $xx=x$ holds). Equivalently, a quandle is an idempotent groupoid $Q$ such that $L_x\in\aut{Q}$ for every $x\in Q$. A quandle is \emph{latin} if it is a quasigroup. We say that a quandle is \emph{involutory} if it is left involutory. Every quandle is flexible and, being idempotent, uniquely $2$-divisible.

The varieties of commutative automorphic loops, left Bruck loops, $\Gamma$-loops and quandles will be denoted by $\A$, $\B$, $\G$, $\Q$, respectively.

\section{Two correspondences}

The following correspondence is well known and likely first appeared in D.A. Robinson's 1964 dissertation. Kikkawa \cite{Kikkawa} published it in 1973, with uniquely $2$-divisible left Bruck loops replaced by ``left diassociative loops satisfying $x(y^2z)=(xy)^2(x^{-1}z)$ in which $x\mapsto x^2$ is a bijection.'' Robinson eventually published it in 1979 \cite{Robinson}, using the terminology of Bruck loops and right distributive right symmetric quasigroups. The correspondence is also alluded to in the recent survey of Stanovsk\'y \cite{Stanovsky}. We give a short but detailed proof that relies only on standard properties of left Bruck loops, summarized in Section \ref{Sc:Notation}.

\begin{theorem}[Kikkawa, Robinson]\label{Th:QBCorrespondence}
Let $Q$ be a set and let $e\in Q$. There is a one-to-one correspondence between involutory latin quandles defined on $Q$ and uniquely $2$-divisible left Bruck loops defined on $Q$ with identity element $e$. In more detail:
\begin{enumerate}
\item[(i)] If $(Q,\cdot)$ is an involutory latin quandle then
\begin{displaymath}
    F_{\Q\to\B}(Q,\cdot)=(Q,+)\text{ defined by }x+y = (x\rdiv e)(e\ldiv y) = (x\rdiv e)(ey)
\end{displaymath}
is a uniquely $2$-divisible left Bruck loop with identity element $e$. Moreover, in $(Q,+)$ we have $-x=ex$, $2x = x+x = xe$ and $x/2 = x/e$.
\item[(ii)] If $(Q,+)$ is a uniquely $2$-divisible left Bruck loop with identity element $e$ then
\begin{displaymath}
    F_{\B\to\Q}(Q,+)=(Q,\cdot)\text{ defined by }xy = (x+x)-y = 2x-y
\end{displaymath}
is an involutory latin quandle.
\item[(iii)] The mappings of \emph{(i)} and \emph{(ii)} are mutual inverses, that is,
\begin{displaymath}
    F_{\Q\to\B}(F_{\B\to\Q}(Q,+)) = (Q,+)
\end{displaymath}
for any uniquely $2$-divisible left Bruck loop $(Q,+)$ with identity element $e$, and
\begin{displaymath}
    F_{\B\to\Q}(F_{\Q\to\B}(Q,\cdot)) = (Q,\cdot)
\end{displaymath}
for any involutory latin quandle $(Q,\cdot)$.
\end{enumerate}
\end{theorem}
\begin{proof}
(i) Let $(Q,\cdot,\ldiv,/)$ be an involutory latin quandle and let $(Q,+)=F_{\Q\to\B}(Q,\cdot)$. Then $(R_e,L_e,1)$ is an isotopism from $(Q,\cdot)$ onto $(Q,+)$ and hence $(Q,+)$ is a loop with identity element $ee=e$. Since $x+ex = (x/e)x = (e/x)x = e$ and $ex+x = e(x/e)e+x = e(x/e)\cdot ex = e(x/e\cdot x) = ee = e$, the two-sided inverse of $x$ in $(Q,+)$ is $-x=ex$. The automorphic inverse property then follows from $-(x+y) = e(x+y) = e(x/e\cdot e\ldiv y) = e(x/e)\cdot y = (ex)/e\cdot y = ex+ey =(-x)+(-y)$. For the left Bol identity, we calculate $xe+(ye+(xe+ez)) = x\cdot e(y\cdot e(xz)) = x(ey\cdot xz) = x(ey)\cdot z = (x(ey)\cdot e)/e\cdot z = (x(ey\cdot xe))/e\cdot z = (x\cdot e(y\cdot e(xe)))/e\cdot z = (xe+(ye+xe))+ez$. Note that $x+x = (x/e)(ex) = (x/e\cdot e)(x/e\cdot x) = xe$. Since $R_e$ is a bijection of $Q$, we see that $(Q,+)$ is uniquely $2$-divisible and $x/2 = x/e$.

(ii) Let $(Q,+)$ be a uniquely $2$-divisible left Bruck loop with identity element $e$ and let $(Q,\cdot) = F_{\B\to\Q}(Q,+)$. Then $(x\mapsto 2x,\,x\mapsto -x,\,1)$ is an isotopism from $(Q,\cdot)$ onto $(Q,+)$ and hence $(Q,\cdot,\ldiv,\rdiv)$ is a quasigroup. We have $xx = 2x-x=x$ by power-associativity, and $x(xy) = 2x-(2x-y) = 2x+(-2x+y) = y$ by the automorphic inverse and left inverse properties. The desired identity $x(yz) = (xy)(xz)$ is equivalent to $2x-(2y-z) = 2(2x-y)-(2x-z)$ and hence to $2x+(-2y+z) = 2(2x-y)+(-2x+z)$. Substituting $x$ for $2x$, $-y$ for $y$, and $x+z$ for $z$ yields $x+(2y+(x+z)) = 2(x+y)+z$, which can be rewritten as $(x+(2y+x))+z = 2(x+y)+z$ by the left Bol identity. We are done by the key identity $x+(2y+x) = 2(x+y)$ for left Bruck loops. (See \cite[Lemma 1]{Glauberman1} or note that the identity follows from $2(x+y) - (x+y) = x+y = x+(2y+(x+(-x-y))) = (x+(2y+x))+(-x-y) = (x+(2y+x))-(x+y)$ upon canceling $x+y$.)

(iii) Let $(Q,+)$ be a uniquely $2$-divisible left Bruck loop with identity element $e$, $(Q,\cdot,\ldiv,/) = F_{\B\to\Q}(Q,+)$ and $(Q,\circ) = F_{\Q\to\B}(Q,\cdot)$. Note that $-y = (e+e)-y = ey$, so $x\circ y = x/e\cdot ey = (x/e + x/e)-ey = (x/e + x/e) + y$. In particular, $x = x\circ e = (x/e+x/e)+e = x/e+x/e$, and we see that $x\circ y = x+y$.

Conversely, let $(Q,\cdot,\ldiv,/)$ be an involutory latin quandle, $(Q,+)=F_{\Q\to\B}(Q,\cdot)$ and $(Q,*) = F_{\B\to\Q}(Q,+)$. In (i) we showed $x+x=xe$ and $-x=ex$, hence $x*y = (x+x)-y = xe-y = (xe)/e\cdot e(-y) = xy$.
\end{proof}

Note that the correspondence of Theorem \ref{Th:QBCorrespondence} is vacuous when $|Q|$ is even.

The following correspondence was proved in \cite{Greer}:

\begin{theorem}[Greer]\label{Th:Greer}
There is a one-to-one correspondence between left Bruck loops of odd order $n$ and $\Gamma$-loops of odd order $n$. In more detail:
\begin{enumerate}
\item[(i)] If $(Q,+)$ is a left Bruck loop of odd order $n$ with identity element $e$ then
\begin{displaymath}
    F_{\B\to\G}(Q,+) = (Q,\cdot)\text{ defined by }x\cdot y = (L_xL_yL_x^{-1}L_y^{-1})^{1/2}L_yL_x(e)
\end{displaymath}
is a $\Gamma$-loop of order $n$. Here, $L_x(y)=x+y$.
\item[(ii)] If $(Q,\cdot)$ is a $\Gamma$-loop of odd order $n$ then
\begin{displaymath}
    F_{\G\to\B}(Q,\cdot) = (Q,+)\text{ defined by }x+y = (x^{-1}\ldiv (y^2x))^{1/2}
\end{displaymath}
is a left Bruck loop of order $n$.
\item[(iii)] The mappings of \emph{(i)} and \emph{(ii)} are mutual inverses, that is,
\begin{displaymath}
    F_{\B\to\G}(F_{\G\to\B}(Q,\cdot)) = (Q,\cdot)
\end{displaymath}
for any $\Gamma$-loop $(Q,\cdot)$ of odd order, and
\begin{displaymath}
    F_{\G\to\B}(F_{\B\to\G}(Q,+)) = (Q,+)
\end{displaymath}
for any left Bruck loop $(Q,+)$ of odd order.
\end{enumerate}
\end{theorem}

\section{Central extensions of loops}\label{Sc:CentralExtensions}

In this section we briefly review the theory of central extensions for loops. Most of the material is well known. We were not able to find Proposition \ref{Pr:CoboundariesAlwaysVCocycles} in the literature; it is a very simple and very useful observation.

Throughout this section, let $F=(F,\cdot,\ldiv,\rdiv,1)$ be a loop and $A=(A,+,0)$ an abelian group.

A loop $Q$ is a \emph{central extension} of $A$ by $F$ if there is a subloop $Z\le Z(Q)$ isomorphic to $A$ such that $Q/Z$ is isomorphic to $F$.

A mapping $\theta:F\times F\to A$ is called a \emph{cocycle}. Given a cocycle $\theta$, define $Q(F,A,\theta)$ on $F\times A$ by
\begin{equation}\label{Eq:Ext}
    (x,a)(y,b) = (xy,\,a+b+\theta(x,y)).
\end{equation}
The resulting groupoid is a quasigroup with $(x,a)\ldiv (y,b) = (x\ldiv y, b-a-\theta(x,x\ldiv y))$ and $(x,a)/(y,b) = (x/y,a-b-\theta(x/y,y))$.

If a cocycle is of the form $\hat\tau(x,y) = \tau(xy)-\tau(x)-\tau(y)$ for some mapping $\tau:F\to A$, then it is called a \emph{coboundary}.

The quasigroup $Q(F,A,\theta)$ is a loop if and only if there is an $a\in A$ such that $\theta(x,1)=\theta(1,x)=-a$ for every $x\in F$, in which case the identity element of $Q(F,A,\theta)$ is $(1,a)$.

\begin{proposition}\label{Pr:CoboundaryIso}
Suppose that $A$ is an abelian group, $F$ a loop, $\theta:F\times F\to A$ a cocycle and $\hat\tau:F\times F\to A$ a coboundary. Then $Q(F,A,\theta)$ is isomorphic to $Q(F,A,\theta+\hat\tau)$.
\end{proposition}
\begin{proof}
Consider the bijection $Q(F,A,\theta)\to Q(F,A,\theta+\hat\tau)$ given by $(x,a)\mapsto (x,a+\tau(x))$.
\end{proof}

In particular, if $Q(F,A,\theta)$ is a loop with identity element $(1,a)$ and $\tau:F\to A$ is any mapping satisfying $\tau(1)=-a$, then $Q(F,A,\theta)$ is isomorphic to the loop $Q(F,A,\theta+\hat\tau)$ with identity element $(1,0)$. We can therefore assume without loss of generality that every cocycle $\theta:F\times F\to A$ satisfies
\begin{equation}\label{Eq:LoopCocycle}
    \theta(x,1)=\theta(1,x)=0
\end{equation}
for every $x\in F$, in which case we call $\theta$ a \emph{loop cocycle}. We note that a coboundary $\hat\tau:F\times F\to A$ is a loop cocycle if and only if it satisfies $\tau(1)=0$, in which case we call it a \emph{loop coboundary}.

Loop cocycles $F\times F\to A$ form a vector space $C(F,A)$ under pointwise addition, and loop coboundaries form a subspace $B(F,A)$ of $C(F,A)$. We have shown that up to isomorphism it suffices to consider representative loop cocycles from the factor space $H(F,A) = C(F,A)/B(F,A)$. In fact, we obtain precisely all central extensions of $A$ by $F$ in this way (see \cite[Theorem 6]{NV64} for a proof):

\begin{theorem}
Let $F$ be a loop and $A$ an abelian group. Then a loop is a central extension of $A$ by $F$ if and only if it is isomorphic to $Q(F,A,\theta)$ for some $\theta\in H(F,A)$.
\end{theorem}

Let $\V$ be a variety of loops. We now address the question when $Q(F,A,\theta)$ belongs to $\V$. A necessary condition for $Q(F,A,\theta)\in\V$ is that both $A\in\V$ and $F\in\V$, since $A\le Q(F,A,\theta)$ and $F$ is a factor of $Q(F,A,\theta)$.

Assuming that $A$, $F\in\V$, we call $\theta\in C(F,A)$ a \emph{$\V$-cocycle} if $Q(F,A,\theta)\in\V$. We denote by $C_\V(F,A)\le C(F,A)$ the set of all $\V$-cocycles.

It is usually straightforward to decide which cocycles are $\V$-cocycles. For instance, if $\V$ is the variety of groups then $\theta$ is a $\V$-cocycle if and only if the group cocycle identity $\theta(x,y)+\theta(xy,z) = \theta(y,z)+\theta(x,yz)$ holds.

The notion of $\V$-cocycles is well-defined on the factor space $H(F,A)$, i.e., there is no need to verify the $\V$-cocycle condition for loop coboundaries:

\begin{proposition}\label{Pr:CoboundariesAlwaysVCocycles}
Let $\V$ be a variety of loops, $A$ an abelian group and $F$ a loop such that $A$, $F\in\V$. Let $\theta\in C(F,A)$ and $\hat\tau\in B(F,A)$. Then $\theta$ is a $\V$-cocycle if and only if $\theta+\hat\tau$ is a $\V$-cocycle.
\end{proposition}
\begin{proof}
The loops $\theta(F,A,\theta)$ and $\theta(F,A,\theta+\hat\tau)$ are isomorphic by Proposition \ref{Pr:CoboundaryIso}.
\end{proof}

Finally, the group $\aut{F}\times\aut{A}$ acts on $C(F,A)$ by
\begin{displaymath}
    \theta\mapsto\theta^{(\alpha,\beta)},\quad
    \theta^{(\alpha,\beta)}(x,y) = \beta^{-1}(\theta(\alpha(x),\alpha(y))).
\end{displaymath}
It also acts on $H(F,A)$ since for any coboundary $\hat\tau$ we have $\hat\tau^{(\alpha,\beta)} = \widehat{\beta^{-1}\tau\alpha}$. Moreover, this action preserves the isomorphism type of the associated loops and hence also the $\V$-cocycle property:

\begin{proposition}
Let $F$ be a loop, $A$ an abelian group, $\theta\in C(F,A)$ and $(\alpha,\beta)\in\aut{F}\times\aut{A}$. Then $Q(F,A,\theta)$ is isomorphic to $Q(F,A,\theta^{(\alpha,\beta)})$.
\end{proposition}
\begin{proof}
Consider the bijection $Q(F,A,\theta^{(\alpha,\beta)})\to Q(F,A,\theta)$ given by $(x,a)\mapsto (\alpha(x),\beta(a))$.
\end{proof}

Altogether, while classifying central extensions of a given abelian group $A$ by a given loop $F$ up to isomorphism, it suffices to consider representatives from the orbits of the group action of $\aut{F}\times\aut{A}$ on $H(F,A)$. It is possible for representatives from distinct orbits to yield isomorphic loops---the isomorphism problem of central extensions is delicate.

\section{Central extensions of left Bruck loops and commutative automorphic loops}

In this section we work out the cocycle conditions for the variety of left Bruck loops and the variety of commutative automorphic loops. We use the same notational conventions as in Section \ref{Sc:CentralExtensions}.

\begin{lemma}\label{Lm:LeftBolCocycle}
Let $F$ be a loop, $A$ an abelian group and $\theta\in C(F,A)$. Then $Q(F,A,\theta)$ is a left Bol loop if and only if $F$ is a left Bol loop and
\begin{equation}\label{Eq:LeftBolCocycle}
    \theta(x,z)+\theta(y,xz)+\theta(x,y(xz)) = \theta(y,x)+\theta(x,yx)+\theta(x(yx),z)
\end{equation}
holds for every $x$, $y$, $z\in F$.
\end{lemma}
\begin{proof}
Straightforward computation shows that $(x,a)((y,b)\cdot (x,a)(z,c))$ is equal to
\begin{displaymath}
    (x(y(xz)), 2a+b+c+\theta(x,z)+\theta(y,xz)+\theta(x,y(xz))),
\end{displaymath}
while $((x,a)\cdot (y,b)(x,a))(z,c)$ is equal to
\begin{displaymath}
    ((x(yx))z,2a+b+c+\theta(y,x)+\theta(x,yx)+\theta(x(yx),z)).
\end{displaymath}
The claim follows.
\end{proof}

\begin{lemma}\label{Lm:AIPCocycle}
Let $F$ be a loop, $A$ an abelian group and $\theta\in C(F,A)$. Then:
\begin{enumerate}
\item[(i)] $Q(F,A,\theta)$ has two-sided inverses if and only if $F$ has two-sided inverses and
\begin{equation}\label{Eq:TwoSidedInversesCocycle}
    \theta(x,x^{-1}) = \theta(x^{-1},x)
\end{equation}
holds for every $x\in F$. Then $(x,a)^{-1} = (x^{-1},-a-\theta(x,x^{-1}))$.
\item[(ii)] $Q(F,A,\theta)$ has the automorphic inverse property if and only if $F$ has the automorphic inverse property, \eqref{Eq:TwoSidedInversesCocycle} holds, and
\begin{equation}\label{Eq:AIPCocycle}
    \theta(x,x^{-1})+\theta(y,y^{-1}) = \theta(x,y)+\theta(x^{-1},y^{-1})+\theta(xy,(xy)^{-1})
\end{equation}
holds for every $x$, $y\in F$.
\end{enumerate}
\end{lemma}
\begin{proof}
(i) The element $(x,a)$ has a two-sided inverse $(y,b)$ if and only if $(1,0)=(x,a)(y,b) = (xy,a+b+\theta(x,y))$ and at the same time $(1,0)=(y,b)(x,a)=(yx,a+b+\theta(y,x))$, that is, if and only if $y=x^{-1}$, $\theta(x,x^{-1})=\theta(x^{-1},x)$ and $b=-a-\theta(x,x^{-1})$. In that case, $(x,a)^{-1} = (x^{-1},-a-\theta(x,x^{-1}))$.

(ii) Suppose that $Q(F,A,\theta)$ has two-sided inverses. Then
\begin{align*}
    (x,a)^{-1}(y,b)^{-1} &= (x^{-1},-a-\theta(x,x^{-1}))(y^{-1},-b-\theta(y,y^{-1}))\\
        &=(x^{-1}y^{-1},-a-b-\theta(x,x^{-1})-\theta(y,y^{-1})+\theta(x^{-1},y^{-1})),
\end{align*}
while
\begin{align*}
    ((x,a)(y,b))^{-1} &= (xy,a+b+\theta(x,y))^{-1}\\
        &= ((xy)^{-1},-a-b-\theta(x,y)-\theta(xy,(xy)^{-1})).
\end{align*}
The claim follows.
\end{proof}

\begin{corollary}\label{Cr:LeftBruckCocycle}
Let $F$ be a loop, $A$ an abelian group and $\theta\in C(F,A)$. Then $Q(F,A,\theta)$ is a left Bruck loop if and only if $F$ is a left Bruck loop and the identities \eqref{Eq:LeftBolCocycle} and \eqref{Eq:AIPCocycle} hold.
\end{corollary}
\begin{proof}
By definition, a loop is left Bruck if and only if it is left Bol and satisfies the automorphic inverse property. By Lemmas \ref{Lm:LeftBolCocycle} and \ref{Lm:AIPCocycle}, $Q(F,A,\theta)$ is left Bruck if and only if $F$ is left Bruck and \eqref{Eq:LeftBolCocycle}, \eqref{Eq:TwoSidedInversesCocycle}, \eqref{Eq:AIPCocycle} hold. Since every left Bol loop has two-sided inverses, the identity \eqref{Eq:TwoSidedInversesCocycle} follows from \eqref{Eq:LeftBolCocycle} and can be omitted.
\end{proof}

Call a loop $Q$ \emph{left automorphic} if $L_{x,y}\in\aut{Q}$ for every $x$, $y\in Q$. As we have already noted in the introduction, a commutative loop is automorphic if and only if it is left automorphic. We obviously have:

\begin{lemma}
Let $F$ be a loop, $A$ an abelian group and $\theta\in C(F,A)$. Then $Q(F,A,\theta)$ is commutative if and only if $F$ is commutative and
\begin{equation}\label{Eq:CommutativeCocycle}
    \theta(x,y)=\theta(y,x)
\end{equation}
holds for every $x$, $y\in F$.
\end{lemma}

\begin{lemma}
Let $F$ be a loop, $A$ an abelian group and $\theta\in C(F,A)$. Then $Q(F,A,\theta)$ is a left automorphic loop if and only if $F$ is a left automorphic loop and
\begin{equation}\label{Eq:LeftAutomorphicCocycle}
\begin{split}
    \theta(x,z)&+\theta(x,u) + \theta(y,xz) + \theta(y,xu) + \theta(yx,L_{x,y}(zu)) + \theta(L_{x,y}(z),L_{x,y}(u))\\
        &= \theta(z,u) +\theta(y,x) + \theta(x,zu) + \theta(y,x(zu)) + \theta(yx,L_{x,y}(z)) + \theta(yx,L_{x,y}(u))
\end{split}
\end{equation}
holds for every $x$, $y$, $z$, $u\in F$.
\end{lemma}
\begin{proof}
Recall that $L_{(x,a)}^{-1}(y,b) = (x,a)\ldiv(y,b) = (x\ldiv y,b-a-\theta(x,x\ldiv y))$ and thus
\begin{align*}
    L_{(x,a),(y,b)}(z,c) &= L_{(y,b)(x,a)}^{-1}L_{(y,b)}L_{(x,a)}(z,c)\\
    &= L_{(yx,a+b+\theta(y,x))}^{-1}L_{(y,b)}(xz,a+c+\theta(x,z)) \\
    &= L_{(yx,a+b+\theta(y,x))}^{-1}(y(xz),a+b+c+\theta(x,z)+\theta(y,xz))\\
    &= (L_{x,y}(z),c+\theta(x,z)+\theta(y,xz)-\theta(y,x)-\theta(yx,L_{x,y}(z))).
\end{align*}
Then $L_{(x,a),(y,b)}(z,c) L_{(x,a),(y,b)}(u,d)$ is equal to $(L_{x,y}(z)L_{x,y}(u),r)$, where $r$ is equal to
\begin{multline*}
     c + d + \theta(x,z) + \theta(y,xz) - \theta(y,x) - \theta(yx,L_{x,y}(z))    \\+ \theta(x,u) + \theta(y,xu) - \theta(y,x) - \theta(yx,L_{x,y}(u)) + \theta(L_{x,y}(z),L_{x,y}(u)),
\end{multline*}
while $L_{(x,a),(y,b)}((z,c)(u,d)) = L_{(x,a),(y,b)}(zu,c+d+\theta(z,u))$ is equal to $(L_{x,y}(zu),s)$, where $s$ is equal to
\begin{displaymath}
    c + d + \theta(z,u) + \theta(x,zu) + \theta(y,x(zu)) - \theta(y,x) - \theta(yx,L_{x,y}(zu)).
\end{displaymath}
The claim follows.
\end{proof}

\begin{corollary}\label{Cr:CommutativeAutomorphicCocycle}
Let $F$ be a loop, $A$ an abelian group and $\theta\in C(F,A)$. Then $Q(F,A,\theta)$ is a commutative automorphic loop if and only if $F$ is a commutative automorphic loop and \eqref{Eq:CommutativeCocycle}, \eqref{Eq:LeftAutomorphicCocycle} hold.
\end{corollary}

\section{The algorithm}

Our approach to enumeration is similar to that of \cite{NV64}. The following algorithm has been implemented in \texttt{GAP} \cite{GAP} using the package \texttt{LOOPS} \cite{LOOPS}.

Let $p$ be an odd prime and $A=\mathbb Z_p$ the cyclic group of order $p$.

\subsection{Central extensions of $A$ by a given factor $F$}

Let $F$ be a loop of order $p^k$. The vector space $B(F,A)$ of loop coboundaries can be constructed as the linear span over the $p$-element field $GF(p)$ of the set $\setof{\widehat{\tau_c}}{1\ne c\in F}$, where $\tau_c:F\to A$ is given by
\begin{displaymath}
    \tau_c(x)=\left\{
    \begin{array}{ll}
        1,&\text{if $x=c$},\\
        0,&\text{otherwise.}
    \end{array}
    \right.
\end{displaymath}

Let now $F$ be a left Bruck loop of order $p^k$. Consider the $|F|^2=p^{2k}$ variables $\theta(x,y)$ indexed by $x$, $y\in F$. By Corollary \ref{Cr:LeftBruckCocycle}, the vector space $C_\B(F,A)$ consists of the solutions to the homogeneous system of $2|F|+|F|^2+|F|^3 = 2p^k+p^{2k}+p^{3k}$ linear equations
\begin{align*}
    \theta(x,1)&=0,  &x\in F,\\
    \theta(1,x)&=0,  &x\in F,\\
    \theta(x,x^{-1}){+}\theta(y,y^{-1}) &= \theta(x,y){+}\theta(x^{-1},y^{-1}){+}\theta(xy,(xy)^{-1}),  &x,\,y\in F,\\
    \theta(x,z){+}\theta(y,xz){+}\theta(x,y(xz)) &= \theta(y,x){+}\theta(x,yx){+}\theta(x(yx), z),  &x,\,y,\,z\in F,
\end{align*}
over $GF(p)$. The linear equations forming this system correspond to the identities \eqref{Eq:LoopCocycle}, \eqref{Eq:LeftBolCocycle} and \eqref{Eq:AIPCocycle}.

When $|F|$ is large enough (say $|F|=3^4$), the linear system must be periodically reduced while it is being set up so as to fit into memory.

For $p=3$ and $k\le 5$, the dimensions of the vector spaces $C_\B(F,A)$ and $B(F,A)$ are recorded in Tables \ref{Tb:1} and \ref{Tb:2}.

The action of $\aut{F}\times\aut{A}$ on $H_\B(F,A)=C_\B(F,A)/B(F,A)$  can be implemented in a straightforward fashion. We were not able to calculate the orbits for the case $F=\Z_3^4$, since $|H_\B(\Z_3^4,\Z_3)| = 3^{24}$.

The set
\begin{displaymath}
    \mathcal Q_\B(F,A) = \setof{Q(F,A,\theta)}{\theta\in H_\B(F,A)\text{ modulo the action of }\aut{F}\times\aut{A}}
\end{displaymath}
contains all left Bruck loops of order $p^{k+1}$ that are central extensions of $A$ by $F$, up to isomorphism. But it can contain duplicate isomorphism types and we must therefore filter $\mathcal Q_\B(F,A)$ up to isomorphism, resulting in a smaller set $\mathcal Q^*_\B(F,A)$.

Calculating $\mathcal Q^*_\B(F,A)$ is a nontrivial task. For instance, there exists a left Bruck loop $F$ of order $3^4$ such that $|\mathcal Q_\B(F,A)|=29525$, so, in the worst case, filtering $\mathcal Q_\B(F,A)$ up to isomorphism will require $\binom{29525}{2}=435848050$ isomorphism checks among loops of order $243$, which is intractable. (It turns out that $|\mathcal Q^*_\B(F,A)|=26865$ here, so the above upper bound is not far from the actual number of isomorphism checks required.)

However, by precalculating certain isomorphism invariants, the set $\mathcal Q_\B(F,A)$ can be pre-partitioned without any isomorphism checks. In more detail, consider $Q\in\mathcal Q_\B(F,A)$. For every $x\in Q$ we have precalculated the numerical invariant $I_x = (I_{x,1},\dots,I_{x,6})$, where
\begin{align*}
    I_{x,1} &= \text{ the cycle structure of $L_x$},\\
    I_{x,2} &= |x|,\\
    I_{x,3} &= ( |\setof{y\in Q}{y^2=x}|,\,|\setof{y\in Q}{y^3=x}|,\,|\setof{y\in Q}{y^4=x}| ),\\
    I_{x,4,a} &= |\setof{y\in Q}{x(xy)=(xx)y\text{ and }|y|=a}|,\\
    I_{x,5,a,b} &= |\setof{(y,z)\in Q\times Q}{y(zx) = (yz)x\text{ and }|y|=a,\,|z|=b}|,\\
    I_{x,6,a} &= |\setof{y\in Q}{xy=yx\text{ and }|y|=a}|.
\end{align*}
(These invariants are not necessarily independent and we have used proper subsets of the invariants in certain situations.) Let  $I(Q)$ be the lexicographically ordered multiset $\setof{I_x}{x\in Q}$. The equivalence relation $\sim$ on $\mathcal Q_\B(F,A)$ defined by $Q_1\sim Q_2$ if and only if $I(Q_1)=I(Q_2)$ induces a partition of $\mathcal Q_\B(F,A)$, and isomorphism checks need to be performed only within each part of the partition. Moreover, given $Q\in\mathcal Q_B(F,A)$, the equivalence relation $\approx$ on $Q$ defined by $x\approx y$ if and only if $I_x=I_y$ induces a partition of $Q$ that must be preserved by any isomorphism from $Q$ to another loop.

Using these invariants, the number of required isomorphism checks in the above example was reduced from $435848050$ to $52475$, which took a few days to perform.

When $F$ is a commutative automorphic loop of order $p^k$, we proceed analogously. The vector space $C_\A(F,A)$ consists of the solutions to the homogeneous system of $2|F|+|F|^2+|F|^4 = 2p^k+p^{2k}+p^{4k}$ linear equations corresponding to the identities \eqref{Eq:LoopCocycle}, \eqref{Eq:CommutativeCocycle} and \eqref{Eq:LeftAutomorphicCocycle}.

\subsection{Central extensions of $A$ by all factors of order $p^k$}

Let $F_1$, $\dots$, $F_m$ be a complete collection of left Bruck loops of order $p^k$ up to isomorphism. Then $\bigcup_{i=1}^m \mathcal Q^*_\B(F_i,A)$ contains all left Bruck loops of order $p^{k+1}$ up to isomorphism, but the union is not necessarily disjoint. To wit, when a constructed loop $Q$ of order $p^{k+1}$ possesses a center of order bigger than $p$, it might also posses two central subloops $Z_1$, $Z_2$ such that $Q/Z_1$, $Q/Z_2$ are not isomorphic.

Fortunately, it is not necessary to perform any additional isomorphism checks among loops of order $p^{k+1}$. Instead, suppose that we would like to decide if a loop $Q\in \mathcal Q^*_\B(F_i,A)$ of order $p^{k+1}$ has been seen before. We calculate the center $Z(Q)$ of $Q$. If $|Z(Q)|=p$ then $Q$ can only be obtained as an extension of $Z(Q)\cong\Z_p$ by $Q/Z(Q)\cong F_i$, and we keep $Q$. Otherwise, we calculate all central subloops $Z_1$, $\dots$, $Z_\ell$ of $Z(Q)$ of order $p$, and we calculate the factors $Q/Z_j$ for $1\le j\le \ell$. If, for some $1\le j\le \ell$, $Q/Z_j$ is isomorphic to $F_t$ with $t<i$, we discard $Q$ since it has already been seen in $\mathcal Q^*_\B(F_t,A)$; otherwise we keep $Q$.

Although this algorithm avoids isomorphism checks among loops of order $p^{k+1}$, it requires a large number of isomorphism checks among loops of order $p^k$ to identity the factor loops $Q/Z_j$. It took several days of computing time to perform this step of the algorithm for left Bruck loops of order $3^5$.

Altogether, the enumeration of left Bruck loops of order $3^k$ with $k\le 5$ took several months of computing time.

Similarly for commutative automorphic loops.

\section{Results}

Our results are summarized in Tables \ref{Tb:1}, \ref{Tb:2} and \ref{Tb:3}. We recall that involutory latin quandles are in one-to-one correspondence with uniquely $2$-divisible left Bruck loops (cf. Theorem \ref{Th:QBCorrespondence}).

\begin{table}[ht]
\begin{displaymath}
\begin{array}{|r|rrrl|}
    \hline
    \text{factor}   &B    &C_\B    &n_\B    &\text{comment}\\
    \hline
    3/1             &1      &2      &2      &\Z_3\\
    \hline
    9/1             &6      &10     &6      &\Z_3^2\\
    9/2             &7      &8      &2      &\Z_9\\
    \hline
    27/1            &23     &34     &47     &\Z_3^3\\
    27/2            &24     &28     &11     &\Z_3\times \Z_9\\
    27/3            &24     &27     &10     &14\text{ elements of order $3$}\\
    27/4            &24     &28     &13     &20\text{ elements of order $3$}\\
    27/5            &24     &27     &6      &2\text{ elements of order $3$}\\
    27/6            &24     &27     &10     &8\text{ elements of order $3$}\\
    27/7            &25     &26     &2      &\Z_{27}\\
    \hline
\end{array}
\end{displaymath}
\caption{The number $n_\B$ of left Bruck loops of order $3^{k+1}$ that are central extensions of $\Z_3$ by a given factor of order $3^k$, up to isomorphism.}\label{Tb:1}
\end{table}

Table \ref{Tb:1} lists all left Bruck loops $F$ of order $3$, $9$ and $27$. For each such loop $F$ of order $3^k$ we give the dimension $B$ of the vector space of coboundaries $B(F,\Z_3)$, the dimension $C_\B$ of the vector space $C_\B(F,\Z_3)$ of left Bruck loop cocycles, and the number $n_\B$ of left Bruck loops of order $3^{k+1}$ up to isomorphism that are central extensions of $\Z_3$ by $F$. The last column of Table \ref{Tb:1} contains structural information that uniquely identifies $F$, either as an abelian group or as a nonassociative left Bruck loop with a given number of elements of order $3$.

As an outcome of this classification, we have observed that given a left Bruck loop of order $3^k\le 81$, the associated $\Gamma$-loop (cf. Theorem \ref{Th:Greer}) is always a commutative automorphic loop.

\begin{table}[ht]
\begin{displaymath}
\begin{array}{|r|rrr|rr||r|rrr|rr|}
\hline
\text{factor}&B&C_\B&n_\B&C_\A&n_\A&\text{factor}&B&C_\B&n_\B&C_\A&n_\A\\
\hline
81/1&76&100&?&100&?&        81/37&77&87&4940&87&4940\\
81/2&77&88&162&88&162&      81/38&77&87&5018&87&5018\\
81/3&77&87&994&87&994&      81/39&77&87&9584&87&9584\\
81/4&77&88&634&88&634&      81/40&77&87&26882&87&26882\\
81/5&77&87&633&87&633&      81/41&77&87&26865&87&26865\\
81/6&77&87&979&87&979&      81/42&77&87&9582&87&9582\\
81/7&77&87&3865&87&3865&    81/43&77&87&9584&87&9584\\
81/8&77&87&7438&87&7438&    81/44&77&87&1169&87&1169\\
81/9&77&87&26913&87&26913&  81/45&77&87&2261&87&2261\\
81/10&77&87&14313&87&14313& 81/46&77&90&260&87&26\\
81/11&77&87&14231&87&14231& 81/47&77&87&189&87&189\\
81/12&77&87&14226&87&14226& 81/48&78&82&11&82&11\\
81/13&77&87&9630&87&9630&   81/49&78&83&17&82&13\\
81/14&77&87&26902&87&26902& 81/50&78&81&8&81&8\\
81/15&77&87&9584&87&9584&   81/51&78&82&6&82&6\\
81/16&77&87&26904&87&26904& 81/52&78&82&13&82&13\\
81/17&77&87&7332&87&7332&   81/53&78&82&7&82&7\\
81/18&77&87&26903&87&26903& 81/54&78&82&16&82&16\\
81/19&77&87&9624&87&9624&   81/55&78&81&10&81&10\\
81/20&77&87&26846&87&26846& 81/56&78&81&10&81&10\\
81/21&77&87&3708&87&3708&   81/57&78&81&8&81&8\\
81/22&77&87&9630&87&9630&   81/58&78&84&36&82&15\\
81/23&77&87&660&87&660&     81/59&78&82&12&81&10\\
81/24&77&87&9637&87&9637&   81/60&78&82&8&80&4\\
81/25&77&87&1759&87&1759&   81/61&78&80&3&80&3\\
81/26&77&87&1759&87&1759&   81/62&78&80&4&80&4\\
81/27&77&87&14228&87&14228& 81/63&78&81&10&81&10\\
81/28&77&87&4940&87&4940&   81/64&78&82&13&82&13\\
81/29&77&87&4986&87&4986&   81/65&78&81&6&81&6\\
81/30&77&87&14227&87&14227& 81/66&78&81&10&81&10\\
81/31&77&87&3822&87&3822&   81/67&78&82&11&81&9\\
81/32&77&87&14226&87&14226& 81/68&78&81&13&81&13\\
81/33&77&87&14239&87&14239& 81/69&78&81&13&81&13\\
81/34&77&87&4938&87&4938&   81/70&78&81&13&81&13\\
81/35&77&87&14218&87&14218& 81/71&78&82&3&80&2\\
81/36&77&87&1928&87&1928&   81/72&79&80&2&80&2\\
\hline
\end{array}
\end{displaymath}
\caption{The number of left Bruck loops $(n_\B)$ and commutative automorphic loops $(n_\A)$ of order $3^5$ that are central extensions of $\Z_3$ by a given factor of order $3^4$, up to isomorphism.}\label{Tb:2}
\end{table}

Table \ref{Tb:2} is similar to Table $1$ but for left Bruck loops of order $81$ used as factors. We were not able to complete the enumeration for the elementary abelian group of order $81$ (the loop $81/1$). There appear to be no compact invariants that could distinguish the $72$ left Bruck loops of order $81$, and we therefore do not give any structural information about the factors.

Unlike for orders $3^k\le 3^4$, there exists a left Bruck loop of order $3^5$ whose associated $\Gamma$-loop is not a commutative automorphic loop. (A multiplication table of this loop can be downloaded from the homepage of the second author.) We therefore report in Table \ref{Tb:2} also data for commutative automorphic loops. Given a left Bruck loop $F$ of order $3^4$, let $G$ be the associated commutative automorphic loop. In the row corresponding to $F$, we give the dimension $C_\A$ of the vector space $C_\A(G,\Z_3)$ of commutative automorphic loop cocycles, and the number $n_\A$ of commutative automorphic loops up to isomorphism that are central extensions of $\Z_3$ by $G$.

Note that for most but not all factors we have $C_\B=C_\A$ and $n_\B=n_\A$. The first difference occurs in the row indexed by the factor $81/46$.

\begin{table}[ht]
\begin{displaymath}
\begin{array}{|r|rrrrr|}
    \hline
    n       &3      &9      &27     &81     &243^*\\
    \hline
    n_\B     &1      &2      &7      &72     &118673^*\\
    n_\A     &1      &2      &7      &72     &118405^*\\
    \hline
\end{array}
\end{displaymath}
\caption{The number of left Bruck loops $(n_\B)$ and commutative automorphic loops $(n_\A)$ of orders $3$, $9$, $27$, $81$, $243$, up to isomorphism, excluding central extensions of $\Z_3$ by $\Z_3^4$.}\label{Tb:3}
\end{table}

Finally, Table \ref{Tb:3} gives the number $n_\B$ of left Bruck loops and the number $n_\A$ of commutative automorphic loops of order $n$ up to isomorphism. For $n=243$, we only give the number of left Bruck loops (resp. commutative automorphic loops) that are \emph{not} central extensions of $\Z_3$ by the elementary abelian group $\Z_3^4$.

It turns out that for each factor $F=81/i$ with $2\le i\le 47$ there is precisely one left Bruck loop up to isomorphism that is a central extension of $\Z_3$ by $\Z_3^4$ (namely the direct product $\Z_3\times F$). The resulting $46$ left Bruck loops are pairwise non-isomorphic and they are \emph{not} included in the count of Table \ref{Tb:3}. For the factors $F=81/i$ with $48\le i\le 72$, no central extension of $\Z_3$ by $F$ is also a central extension of $\Z_3$ by $\Z_3^4$. The situation is completely analogous for commutative automorphic loops.

Therefore, if $N_\B$ (resp. $N_\A$) is the number of left Bruck loops (resp. commutative automorphic loops) of order $243$ up to isomorphism that are central extensions of $\Z_3$ by $\Z_3^4$, then the number of left Bruck loops (resp. commutative automorphic loops) of order $243$ up to isomorphism is $N_\B+118673$ (resp. $N_\A+118405$).

\section{Open problems}

Let $p$ be an odd prime. In \cite{Greer}, Greer asked if the $\Gamma$-loops associated with left Bruck loops of order $p^3$ are always commutative automorphic loops. We can generalize his question as follows:

\begin{problem}
For which odd primes $p$ and positive integers $k$ is there a one-to-one correspondence between left Bruck loops of order $p^k$ and commutative automorphic loops of order $p^k$?
\end{problem}

By the results mentioned in the introduction, the answer is positive when $k\le 2$. Our results imply that the answer is positive for $p^k\in\{3^3,3^4\}$ and negative for $p^k=3^5$. We have also verified that the answer is positive for $p^k\in\{5^3,7^3,11^3\}$, the case $11^3$ taking several days of computing time to complete.

\begin{problem}
Let $p$ be an odd prime and $k$ a positive integer. Is there an abstract description of left Bruck loops of order $p^k$ for which all associated $\Gamma$-loops are commutative automorphic loops?
\end{problem}

\section*{Acknowledgment}

We thank David Stanovsk\'y for bringing the correspondence between involutory latin quandles and uniquely $2$-divisible left Bruck loops to our attention. We also thank an anonymous referee for several comments that improved the manuscript. The calculations took place on the high performance computing cluster of the University of Denver---we thank Benjamin Fotovich for providing assistance with the computing cluster.

\bibliographystyle{amsplain}

\end{document}